\documentclass[12pt,reqno,twoside]{amsart}
\pdfoutput=1
\usepackage{amssymb}
\usepackage{amsmath,accents}
\usepackage[latin1]{inputenc}
\usepackage[dvips]{color}
\usepackage{graphicx}
\usepackage[all]{xy}
\usepackage{pstricks, pst-3d}
\def\ccn{^{\hspace{0.02cm}\circ\hspace{.04cm}\circ}}

\def\nc{^{\hspace{0.21cm}\circ}}

\DeclareMathOperator{\pnt}{\raise 0.5mm \hbox{\large\bf.}}

\def\+#1{\relax\ifmmode\if\noexpand #1\relax \mathop{\kern
   0pt^+{#1}}\nolimits\else \kern 0pt^+\!#1 \fi\else$^*$#1\fi}

\let\:=\colon
\definecolor{vio}{rgb}{0.8,0.2,0.8}
\newtheorem{thm}{\bf Theorem}[section]

\newtheorem{prop}[thm]{\bf Proposition}

\newtheorem{property}[thm]{\bf Property}

\theoremstyle{definition}

\newtheorem{rem}[thm]{\bf Remark}

\theoremstyle{plain}
\newtheorem*{thm*}{Theorem}
\newtheorem*{quest*}{Question}

\begin{document}
\today
\title{Column tessellations}
\author{Ngoc Linh Nguyen, Viola Weiss and Richard Cowan}
\address{Friedrich-Schiller-Universit\"at Jena, Institut füt Stochastik, Germany}
\address{Ernst-Abbe-Fachhochschule Jena, Fachbereich Grundlagenwissenschaften, Germany}
\address{School of Mathematics and Statistics, University of Sydney, Australia}
\email{linh.nguyen@uni-jena.de}
\email{viola.weiss@fh-jena.de}
\email{richard.cowan@sydney.edu.au}

\thanks{The first and second author have been supported by Deutsche Forschungsgemeinschaft.}

\begin{abstract}
A new class of random spatial tessellations is introduced -- the so-called column tessellations of three-dimensional space. The construction is based on a stationary planar tessellation.  Each cell of the spatial tessellation is a prism whose base facet is  congruent to a cell
of the planar tessellation. Thus intensities, topological and metric mean values of the spatial tessellation can be calculated by suitably chosen parameters of the planar tessellation. A column tessellation is not facet-to-facet.
\end{abstract}

\maketitle
\bigskip

\section{Introduction}
\label{intro}

Random tessellations are one of the classical structures considered in stochastic geometry. Two standard models are the Poisson hyperplane and Poisson Voronoi tessellations, see \cite{SW}, \cite{SKM}. In the plane these tessellations are side-to-side. That means each side of a polygonal tessellation cell coincides with a side of a neighbouring cell. In higher dimensions they are  facet-to-facet. In recent years there has been a growing interest in tessellation models that do not fulfill this property. A first systematic study of the effects when a tessellation is not facet-to-facet is given in \cite{WC} for the planar and spatial case, with a further planar study presented in \cite{CT}. Tessellations of that kind arise for example by subsequent cell division. Among these models the \emph{iteration stable} or STIT tessellations are of particular interest, because of the number of analytically available results, see \cite{NW05}, \cite{MNW08}, \cite{CO}, \cite{TWN}, \cite{TW} and the references therein. They may serve as a reference model for crack and fissure structures or for processes of cell division. The development of new model classes is important for further applications to random structures in materials science, geology and biology  --- and the current paper contributes to that aim.

In this paper we consider a new class of spatial tessellations, whose construction is based on a stationary planar tessellation ${\mathcal Y}'$. From  each cell  $z$ of ${\mathcal Y}'$ we form an infinite column perpendicular to the plane $\mathcal{E}$ in which ${\mathcal Y}'$ lies and having that planar cell $z$ as cross-section. To create a spatial tessellation, each infinite column is intersected by planar plates which are congruent to $z$ and parallel to $\mathcal{E}$. Thus the spatial cells which arise are prisms and their base facets are translations (in the third dimension orthogonal to $\mathcal{E}$) of the cells of ${\mathcal Y}'$. The resulting three-dimensional tessellation ${\mathcal Y}$ is called a \emph{column tessellation}.  The intersecting plates of a column are positioned so that no plate is coplanar with a plate of a neighbouring column. Hence cells in neighbouring columns do not have a common facet. Therefore the tessellation ${\mathcal Y}$ is not facet-to-facet. The definitions of how the plates intersect the columns can vary, thus giving scope to consider different cases -- and so to construct a rich model class. The column tessellations we study are a generalization of less general column constructions considered in \cite{WC} and a modification of stratum mosaics introduced by Mecke \cite{M84}. Column tessellations could be useful to describe crack structures in geology, as for example in the Giant's Causeway of Northern Ireland (see Figure \ref{gc}).

 \begin{figure}[h]
    \begin{center}
        \includegraphics[width=80mm]{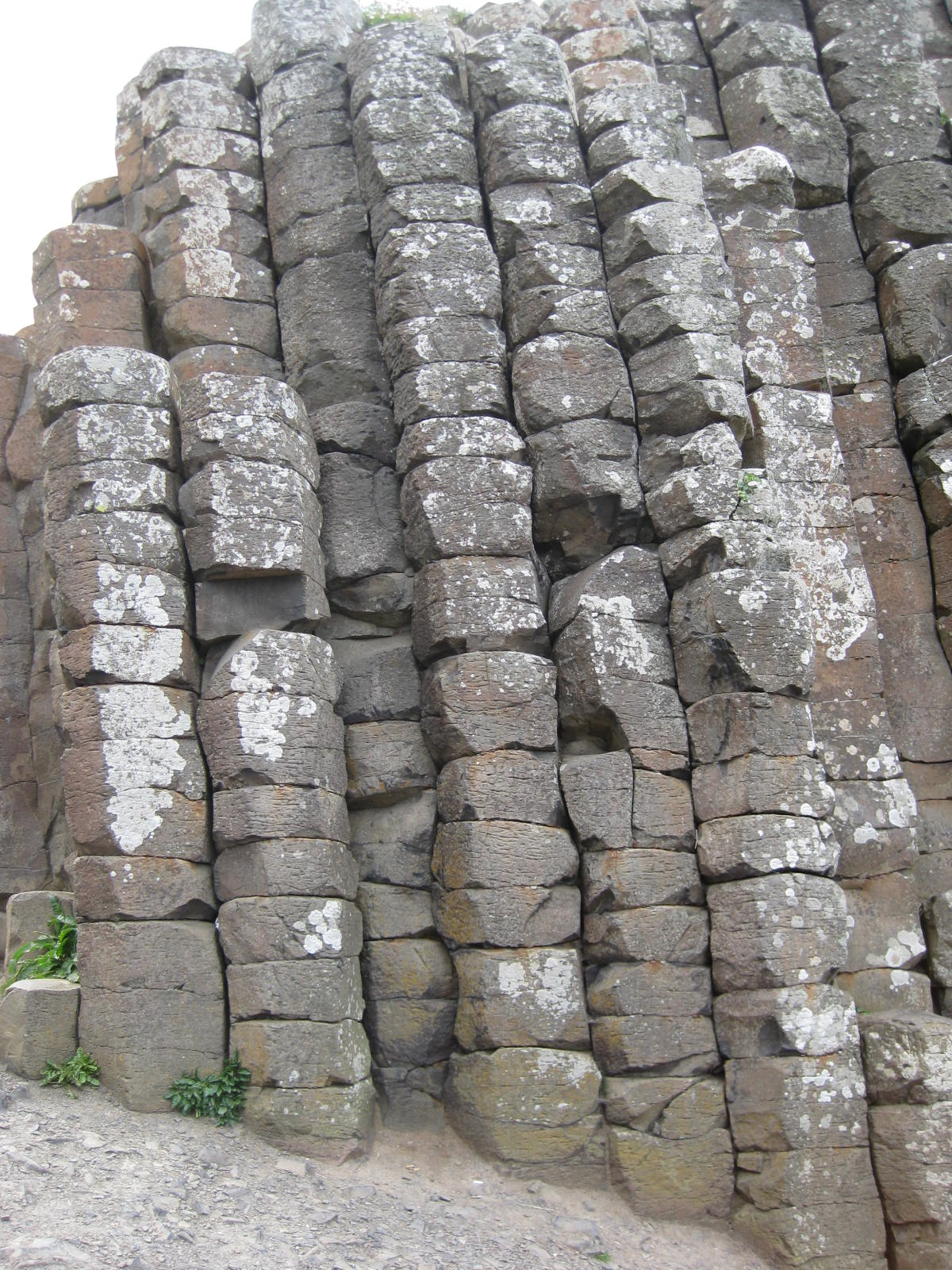}\\
        \caption{ \label{gc} Basalt columns, approximately $6-8$ metres high, divided by `plates' at approximately $30$ cm spacing. Photo taken by one of the authors. There are many formations like this one (at the Giant's Causeway, Northern Ireland) around the world. }
    \end{center}
\end{figure}

In this paper we will explore the question of which parameters of the planar tessellation are necessary to calculate characteristics of the spatial tessellation. This is interesting, for example, when only a planar section through a spatial column tessellation can be observed.

The paper is organized as follows. To describe in detail the topological effects, for tessellations which are not facet-to-facet, we use the system of notations given in \cite{WC}. Section $\ref{notations}$ gives a short introduction to basic notations of planar and spatial tessellations.  In Section $\ref{generalmodel}$ the general construction of a column tessellation is explained, special notations are defined and basic properties are considered. For reasons of comprehensibility throughout the paper we often consider the special case, where the intersecting plates in a column have constant separation $1$. This generates a column tessellation where all the cells have height $1$. We illustrate notations and results for this special case. In Section $\ref{relations}$ it is shown that intensities and topological mean values of a spatial column tessellation can be determined by suitably chosen parameters of the planar tessellation. Later in this section relations for metric mean values of the column tessellation are also deduced.

\section{Basic notations}
\label{notations}
In this paper we study stationary random tessellations in $\mathbb{R}^2$ and $\mathbb{R}^3$ having only convex cells. We use the system of notation given in \cite{WC}.
For a spatial tessellation (that is, of $\mathbb{R}^3$), we deal with four kinds of {\it primitive elements}: vertices, edges, plates and cells. The corresponding classes are denoted by $\mathsf{V}$, $\mathsf{E}$, $\mathsf{P}$ and $\mathsf{Z}$. The primitive elements are $k$-dimensional convex polytopes, $k=0,1,2,3$, which cannot have any other elements in their relative interior. An object belonging to a class $\mathsf{X}$ is often referred to as ``an $\mathsf{X}$-type object" or ``an object of type $\mathsf{X}$".

The {\it intensity} of objects of class $\mathsf{X}$ is denoted by $\lambda_\mathsf{X}$. It is the mean number of centroids of $\mathsf{X}$-type objects per unit volume. It is assumed henceforth that $0 < \lambda_\mathsf{X} < \infty$; this is the case for all example tessellations considered.  Recall that an object $x$ of $\mathsf{X}$ is said to be {\it adjacent} to an object $y$ of $\mathsf{Y}$ if either $x\subseteq y$ or $y\subseteq x$. Let $\mu_\mathsf{XY}$ be the mean number of $\mathsf{Y}$-type objects adjacent to the typical object of $\mathsf{X}$. Formally the {\it typical} object of class $\mathsf{X}$ can be introduced by means of Palm distributions for which we refer to \cite{SW}, \cite{SKM}. Intuitively it can be considered as a uniformly selected object from $\mathsf{X}$ independent of its size and shape. For an element $x \in \mathsf{X}$ the number of $\mathsf{Y}$-type objects adjacent to $x$ is denoted by $m_\mathsf{Y}(x)$. Formally, we write $\mu_\mathsf{XY}:=\mathbb{E}_\mathsf{X}[m_\mathsf{Y}(x)]$, where $\mathbb{E}_\mathsf{X}$ denotes an expectation for the typical object of type $\mathsf{X}$ with respect to the Palm measure.

Because of combinatorial and topological relations within a spatial tessellation the twelve adjacency mean values  $\mu_\mathsf{XY}$, for  $\mathsf{X}$ and $\mathsf{Y}$ $\ \in \{\mathsf{V},\mathsf{E},\mathsf{P},\mathsf{Z}\}$ and $\mathsf{X} \not= \mathsf{Y}$ can be expressed as functions of three {\it cyclic adjacency parameters}\\[2mm]
\hspace*{5mm}
\begin{tabular}{lcl}
$\mu_\mathsf{VE}$ & -- &  the mean number of edges emanating from the typical vertex, \\
$\mu_\mathsf{EP}$ & -- & the mean number of plates emanating from the typical edge and \\
$\mu_\mathsf{PE}$ & -- & the mean number of vertices on the boundary of the typical \\ &&plate,
\end{tabular}\\[2mm]
see \cite{WC} (or also \cite{rad} and  \cite{M84} where another notation is used).

In the facet-to-facet case, the $k$-dimensional faces of an $\mathsf{X}$-type object are primitive elements. In contrast for non facet-to-facet tessellations we must carefully distinguish between the primitive elements and the k-faces of polytopes. For example a cell can have vertices on its boundary which are not $0$-faces of that polytope. A $1$-face (ridge) of a cell can have vertices in its relative interior, this is impossible for edges. A $2$-face (facet) of a cell  may not be a plate. Hence we use the notation $\mathsf{X_k}$ for the class of all $k$-faces of $\mathsf{X}$-type polytopes, $k<\dim (\mathsf{X}$-object). For instance $\mathsf{P_1}$ is the class of the $1$-dimensional faces of all plates, called the plate-sides. We emphasize that some of these classes are {\it multisets} because of the multiplicities of the elements. For example if a vertex $v$ is a 0-face of $j$ cells (note that $j \leq m_\mathsf{Z}(v)$) then the class $\mathsf{Z_0}$ has $j$ elements equal to $v$. Furthermore, we define $n_k(x)$ as the number of $k$-faces of a particular object $x\in \mathsf{X}$ and $\nu_k(\mathsf{X}):=\mathbb{E}_\mathsf{X}[n_k(x)]$ is the mean number of $k$-faces of the typical $\mathsf{X}$-object. For example it is\\[2mm]
\hspace*{5mm}
\begin{tabular}{lcl}
$\nu_0(\mathsf{P})$ & -- &  the mean number of $0$-faces of the typical plate, \\
$\nu_1(\mathsf{Z})$ & -- & the mean number of $1$-faces (ridges) of the typical cell.
\end{tabular}\\[2mm]
Sometimes we use $\mathsf{X[.]}$ for a subset of the class $\mathsf{X}$, where the term in the brackets is a suitable chosen symbol describing the property of the subclass. For example, the subclasses of horizontal and vertical edges are denoted by $\mathsf{E[hor]}$ and $\mathsf{E[vert]}$.

If a tessellation is not facet-to-facet, a face of a primitive element can have interior structure. To quantify the effects of this phenomenon four additional parameters are introduced in \cite{WC}, called {\it interior parameters} and defined as follows:\\[2mm]
\hspace*{5mm}
\begin{tabular}{lcl}
$\xi$ & -- & the proportion of edges whose interiors are contained in the interior \\ && of some cell-facet, \\
$\kappa$ & -- & the proportion of vertices in the tessellation contained in the interior \\ && of some cell-facet, \\
$\psi$ & -- & the mean number of ridge-interiors adjacent to the typical vertex, \\
$\tau$ & -- & the mean number of plate-side-interiors adjacent to the typical vertex.
\end{tabular}\\[2mm]
Note that the interior parameters using the adjacency notation can be written as $$\xi=\mu_\mathsf{EZ_2}\ccn,\ \
\kappa=\mu_\mathsf{VZ_2}\nc, \ \
\psi=\mu_\mathsf{VZ_1}\nc \ \ {\rm  and} \ \
\tau=\mu_\mathsf{VP_1}\nc,$$
using $\mathsf{\accentset{\circ}{X}}$ for the class of relative interiors of members of $\mathsf{X}$. We call an edge whose interior is contained in the interior of a cell-facet a {\it $\pi$-edge} and a vertex in the interior of a cell-facet is a {\it hemi-vertex}, see \cite{WC}.

Naturally all four interior parameters are zero in the facet-to-facet case. In \cite{CW} it is shown that a spatial tessellation is facet-to-facet with probability 1 if and only if $\xi =0$.

Some further notations will be given later.

The initial point of the construction of a column tessellation is a stationary planar tessellation ${\mathcal Y'}$ in a fixed plane $\mathcal{E}$ which, without loss of generality, is assumed horizontal.  The classes  of \emph{planar} primitive elements of ${\mathcal Y}'$ are $\mathsf{V}$ (vertices), $\mathsf{E} $ (edges) and $\mathsf{Z} $ (cells). Their intensities $\lambda'_\mathsf{X}$ and the adjacency mean values $\mu'_\mathsf{XY}$, $\mathsf{X},\mathsf{Y} \in \{\mathsf{V},\mathsf{E},\mathsf{Z}\}$, are marked with a prime. A planar tessellation which is not side-to-side has vertices located in the interior of  cell-sides. We call them {\it $\pi$-vertices}, because one angle created by the emanating edges is equal to $\pi$. The {\it interior parameter} of a planar tessellation is \\[2mm]
\hspace*{5mm}
\begin{tabular}{lcl}
$\phi$ & -- & the proportion of $\pi$-vertices in the tessellation, $\phi = \mu'_\mathsf{V\accentset{\circ}{Z}_1}$.
\end{tabular}

\section{Column tessellations}
\label{generalmodel}
\subsection{Construction}
Based on the planar tessellation ${\mathcal Y}'$ in $\mathcal{E}$ we construct the spatial column tessellation ${\mathcal Y}$ in the following way:

For each cell $z$ of ${\mathcal Y}'$, we consider an infinite cylindrical column based on this cell and perpendicular to $\mathcal{E}$.  Further we mark $z$'s centroid with a real-valued positive $\rho_z$. Here $\rho_z$ is a non-random function of some aspects of $\mathcal{Y'}$ viewed from $z$,  perhaps the size, shape or environment of the cell $z$, say. Such a mark is created for all cells in ${\mathcal Y}'$. Now, for each planar cell $z$, we construct on the line   going through the cell-centroid of $z$ and perpendicular to $\mathcal{E}$ a stationary  point processes with intensity $\rho_z$. The point processes on different lines are  conditionally independent given the information in the planar tessellation ${\mathcal Y}'$. To create the spatial tessellation, a column based on $z$ is intersected by horizontal plates, one of these containing each of the random points of that column's point process. The resulting tessellation ${\mathcal Y}$ is called {\em column tessellation}.  Note that the lines through the cell-centroids do not belong to the column tessellation. Any cell of ${\mathcal Y}$ is a right prism, where its base facet is a vertical translation of a cell of ${\mathcal Y}'$. Because of the conditional independence, there are no coincidences among the horizontal plates that appear in different columns, and so (with probability $1$) the cells in neighbouring columns do not have a common facet. Hence a column tessellation is not facet-to-facet. The intersection of  a column tessellation ${\mathcal Y}$ with any plane parallel to $\mathcal{E}$ is a vertical translation of ${\mathcal Y}'$.

A simple case of this general construction is when we take $\rho_z=1$, a constant for all cells $z$ of ${\mathcal Y}'$. For a column let $\zeta_k$, $k = 0, \pm 1, \pm 2,...$, be the random distances of the intersection planes from $\mathcal{E}$, then $\zeta_0$ is uniformly distributed in $[0,1]$ and $\zeta_{k+1}=\zeta_k + 1$ for all $k$. The positions of the cuts in a column are stationary  and \emph{completely} independent of the cuts in the neighbouring cylinders, as no information has been drawn from ${\mathcal Y}'$. Any cell of the column tessellation ${\mathcal Y}$ that has arisen is a right prism with height $1$. For short, we call it a \emph{column tessellation with height 1}. An example is given in Figure \ref{figcoltess}. On the top on the left the planar tessellation ${\mathcal Y}'$ is shown and the columns formed by the cells of ${\mathcal Y}'$ on the right. On the bottom left we see the columns with the cuts generated by the parallel horizontal plates, using three different colors for three columns. Down the right we strike ${\mathcal Y}'$ off because it is not a part of the column tessellation ${\mathcal Y}$.

\begin{figure}[top]
\centering
\includegraphics[scale=0.4]{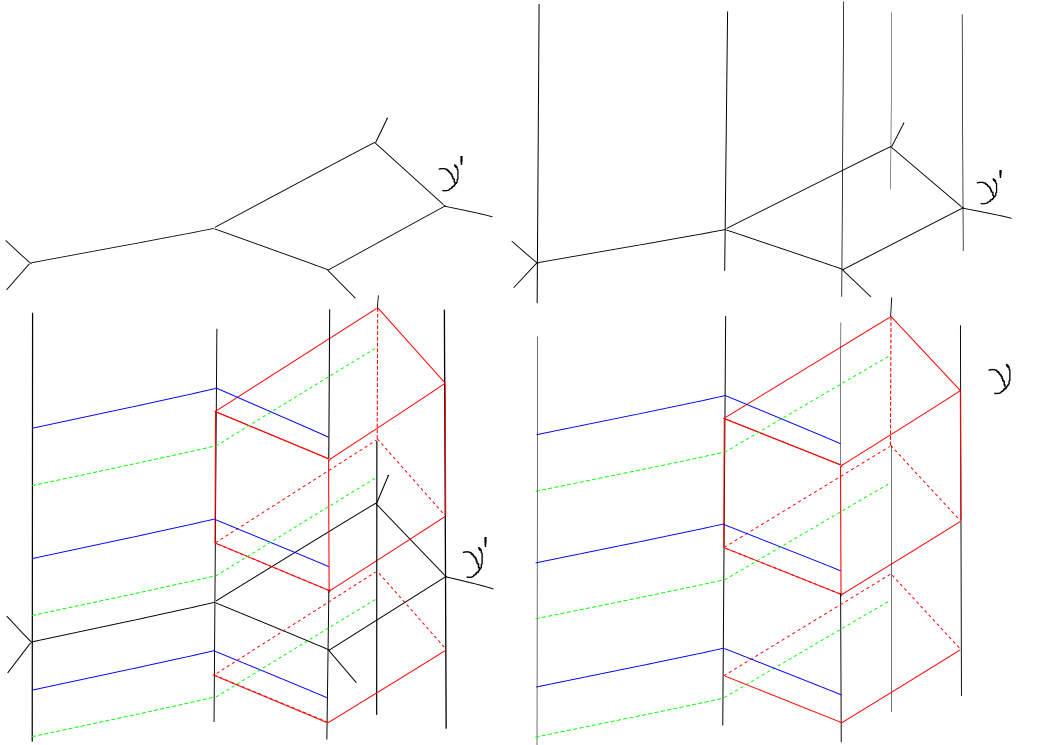}
\caption{Column tessellation ${\mathcal Y}$ with constant height $1$}
\label{figcoltess}
\end{figure}

\subsection{Notations}
Besides the basic notations given in Section \ref{notations}, we need further notations for planar tessellations. Some of these are based on a relationship between cells and lower dimensional objects of the planar tessellation - an {\it ownership relation}. We describe the ownership relation using a function $b$ ({\it belonging to}) as follows. A cell $z$ is the owner $z=b(z_j)$ of an element $z_j \in \mathsf{Z}_j$, if $z_j$ is a j-face of that cell $z$, $j=0,1$. It is obvious that $z$ is the owner of $n_j(z)$ j-faces and that any $z_j \in \mathsf{Z}_j$ has its unique owner. Furthermore we are interested in the vertices of a cell which are not corners (0-faces) of that cell. It is obvious that those vertices are $\pi$-vertices. We say that $z$ is the owner of such a $\pi$-vertex which is not a $0$-face of $z$. Thus any $\pi$-vertex $v[\pi]$ belongs to a unique owner-cell $z=b(v[\pi])$ and a cell $z$ owns $m_\mathsf{V}(z)-n_0(z)$ $\pi$-vertices. So our \emph{belongs to} function $b$ has domain $\mathsf{Z_0}\cup \mathsf{Z_1} \cup \mathsf{V}[\pi]$ and range $\mathsf{Z}$.\\

Ensuing from the mark $\rho_z$ of a cell $z$ we define the following notations for the planar tessellation ${\mathcal Y}'$:
\begin{itemize}
\item based on the planar adjacency relationship `$x$ is adjacent to $z \in \mathsf{Z}$'

$\alpha_x = \sum\limits_{\{z:z \supset x\}} \rho_z$, \\ where we later mostly consider the cases  $x=v \in \mathsf{V}$, $x=e \in \mathsf{E}$ and $x=v[\pi] \in \mathsf{V}[\pi]$,
\newpage
\item based on the ownership relation `$z \in \mathsf{Z}$ owns $z_0 \in \mathsf{Z}$ or $v[\pi] \in \mathsf{Z}$',

$\beta_{z_0} = \rho_{b(z_0)}$,

$\beta_{v[\pi]}=\rho_{b(v[\pi])}$, \\[2mm] where  $z_0$ is  a $0$-face of the cell $b(z_0)$ and $v[\pi]$ is a $\pi$-vertex and no $0$-face of the cell $b(v[\pi])$ and,
\item based on a weighting,

$\gamma_v =m'_\mathsf{Z}(v) \alpha_v,$ \quad (number-weighted)

$\gamma_e =\ell'(e)\alpha_e$, \quad (length-weighted)

$\gamma_z =a'(z)\rho_z$,  \quad (area-weighted)\\[2mm]
where  $\ell'(e)$ is the length of the edge $e$ and $a'(z)$ is the area of the cell $z$.
\end{itemize}

\begin{figure}[top]
\centering
\includegraphics[scale=0.24]{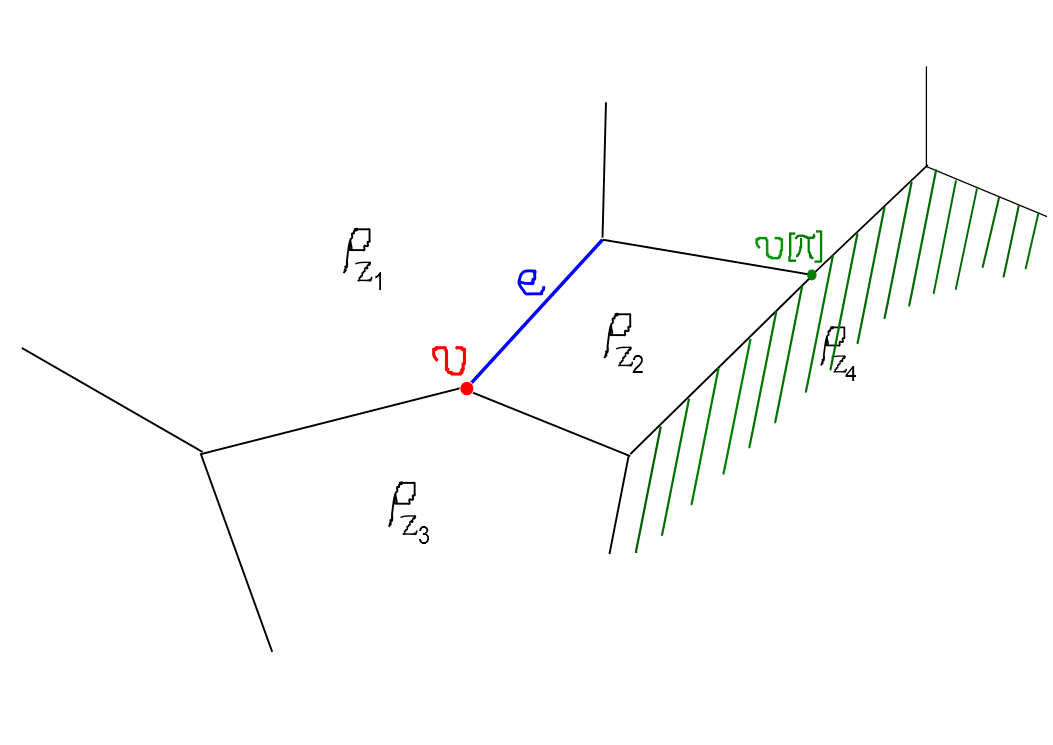}
\caption{An example of adjacency and ownership relation in the planar tessellation $\mathcal{Y}'$}
\label{adjrel}
\end{figure}	

Figure \ref{adjrel} illustrates an example for these notations and the differences between ownership and adjacency relation. The vertex $v$ is adjacent to the cells $z_1$, $z_2$ and $z_3$, the edge $e$ is adjacent to the cells $z_1$ and $z_2$, hence $\alpha_v=\rho_{z_1}+\rho_{z_2}+\rho_{z_3}$ and $\alpha_e=\rho_{z_1}+\rho_{z_2}.$ For the ownership relation, it is easy to see that for the $\pi$-vertex $v[\pi]$ we have $\beta_{v[\pi]}=\rho_{z_4}$,  because $z_4=b(v[\pi])$. Besides, $v$ is a $0$-face of the cells $z_1$, $z_2$ and $z_3$, then the class $\mathsf{Z_0}$ has $3$ elements equal to $v$ denoted by $z_{01}$, $z_{02}$ and $z_{03}$ with owner-cells $z_1$, $z_2$ and $z_3$, respectively. Hence $\beta_{z_{01}}=\rho_{z(z_{01})}=\rho_{z_1}$, $\beta_{z_{02}}=\rho_{z_2}$ and $\beta_{z_{03}}=\rho_{z_3}$.

All these $\alpha$-, $\beta$-, $\gamma$-quantities can be understood as marks of elements of the planar tessellation. Each of these marks leads to mark distributions. The corresponding means are:

\begin{itemize}
\item $\bar{\rho}_\mathsf{Z}=\mathbb{E}'_\mathsf{Z}(\rho_z)$ \ -- \ the mean $\rho$-intensity of the typical cell,
\item $\bar{\alpha}_\mathsf{X}=\mathbb{E}'_\mathsf{X}(\alpha_x)$ \ -- \ the mean {\bf total} $\rho$-intensity of all cells adjacent to the typical $\mathsf{X}$-object,
\item $\bar{\beta}_{\mathsf{Z_0}}=\mathbb{E}'_{\mathsf{Z_0}}(\beta_{z_0})$  and   $\bar{\beta}_\mathsf{V[\pi]}=\mathbb{E}'_\mathsf{V[\pi]}(\beta_{v[\pi]})$ \ -- \ the mean $\rho$-intensity of the owner cell of the typical $0$-face or the typical $\pi$-vertex, respectively, \\
and
\item $\bar{\gamma}_\mathsf{X}=\mathbb{E}'_\mathsf{X}(\gamma_x)$ \ -- \ the mean total weighted $\rho$-intensity of all cells adjacent to the typical $\mathsf{X}$-object.
\end{itemize}
\begin{rem} \label{rem:second-order}
Using mean value identities for tessellations given in \cite{MO} and  some generalizations derived recently in \cite{wnc}, most of the above mean values can be expressed as second-order quantities depending on the $\rho$-intensity as follows:
\begin{itemize}
\item $\lambda'_\mathsf{X}\bar{\alpha}_\mathsf{X}=\lambda'_\mathsf{Z}\mathbb{E}'_\mathsf{Z}(m'_\mathsf{X}(z)\rho_z),$
\item $\lambda'_{\mathsf{Z_0}}\bar{\beta}_{\mathsf{Z_0}}=\lambda'_\mathsf{Z}\mathbb{E}'_\mathsf{Z}(n'_0(z)\rho_z)$,
\item $\lambda'_\mathsf{V[\pi]}\bar{\beta}_\mathsf{V[\pi]}=\lambda'_\mathsf{Z}\mathbb{E}'_\mathsf{Z}[(m'_\mathsf{V}(z)-n'_0(z))\rho_z]=\lambda'_\mathsf{V}\bar{\alpha}_\mathsf{V}-\lambda'_{\mathsf{Z_0}}\bar{\beta}_{\mathsf{Z_0}}$,
\item $\lambda'_\mathsf{E}\bar{\gamma}_\mathsf{E}=\lambda'_\mathsf{Z}\mathbb{E}'_\mathsf{Z}(\ell'(z)\rho_z)$, where $\ell'(z)$ is the perimeter of the planar cell $z$,
\item $\lambda'_\mathsf{Z}\bar{\gamma}_\mathsf{Z}=\lambda'_\mathsf{Z}\mathbb{E}'_\mathsf{Z}(a'(z)\rho_z)$.
\end{itemize}
Only $\bar{\gamma}_\mathsf{V}$ requires a separate argument: $$\lambda'_\mathsf{V}\bar{\gamma}_\mathsf{V}=\lambda'_\mathsf{Z}\mathbb{E}'_\mathsf{Z}[(k'_\mathsf{E}(z) +m'_\mathsf{V}(z))\rho_z]=\lambda'_\mathsf{Z}\mathbb{E}'_\mathsf{Z}(k'_\mathsf{E}(z)\rho_z)+\lambda'_\mathsf{V}\bar{\alpha}_\mathsf{V},$$ where $k'_\mathsf{E}(z):=\sum\limits_{e\in\mathsf{E}}{\bf1}\{e\cap z\ne\emptyset\}$ -- the number of edges intersecting $z$.
\end{rem}
\begin{rem}
To illustrate these mean values we consider now the special case when $\rho_z=1$ for all $z\in\mathsf{Z}$.

\label{rem:intensity1}
\begin{align*}
\bar{\rho}_\mathsf{Z}&=1,\\
\bar{\alpha}_\mathsf{V}&=\mu'_\mathsf{VZ}=\mu'_\mathsf{VE},\\
\bar{\alpha}_\mathsf{E}&= 2, \\
\bar{\alpha}_{\mathsf{V[\pi]}}&= \mu'_\mathsf{V[\pi]Z} = \mu'_\mathsf{V[\pi]E},\\
\bar{\beta}_{\mathsf{Z_0}}&=1,\\
\bar{\beta}_{\mathsf{\mathsf{V[\pi]}}}&=1,\\
\bar{\gamma}_\mathsf{V}&=\mu'^{(2)}_\mathsf{VZ}=\mu'^{(2)}_\mathsf{VE},\\
\bar{\gamma}_\mathsf{E}&=2\bar{\ell}'_\mathsf{E},\\
\bar{\gamma}_\mathsf{Z}&=\bar{a}'_\mathsf{Z},
\end{align*}
where \\
\hspace*{5mm}
\begin{tabular}{lcl}
$\mu'_\mathsf{V[\pi]E}$ & -- & the mean number of emanating edges from the typical $\pi$-vertex, \\
$\mu'^{(2)}_\mathsf{VE}$ & -- & the second moment of the mean number of edges adjacent to the \\ && typical vertex, \\
$\bar{\ell}'_\mathsf{E}$ & -- & the mean length of the typical edge, \\
$\bar{a}'_\mathsf{Z}$ & -- & the mean area of the typical cell.
\end{tabular}\\[2mm]
Formally, the second moment of the mean number of edges adjacent to the typical vertex is given by $\mathbb{E}'_\mathsf{V}[m'_\mathsf{E}(v)^2]$, where $\mathbb{E}'_\mathsf{V}$ denotes an expectation for the typical vertex of ${\mathcal Y}'$ with respect to the Palm measure, see \cite{SKM}.
\end{rem}
The first and the last of the above relations are obvious. The relations for the three $\alpha$-means follow from $\alpha_x = m'_\mathsf{Z}(x)$ in the case $\rho_z=1$. The $ \beta$-mean value relations arise from that fact that the owner-cell of the typical corner or the typical $\pi$-vertex, respectively, has $\rho$-intensity 1. And the first two $\gamma$-mean values we obtain using again $\alpha_x = m'_\mathsf{Z}(x)$ for $x=v$ and $x=e$.

Considering again the general construction, our aim is the calculation of intensities and mean values of the column tessellation ${\mathcal Y}$ from the characteristics of ${\mathcal Y}'$. For this purpose the following basic relations between vertices and edges of ${\mathcal Y}$ and ${\mathcal Y}'$ are helpful.

\subsection{Basic properties}
For a vertex $v \in \mathsf{V}$ in ${\mathcal Y}'$ we consider the vertical line $\mathcal{L}_v$ through $v$  and its intersection with the columns created by the planar cells adjacent to $v$. The horizontal plates in these columns create a point process  (comprising vertices of the spatial tessellation $\mathcal{Y}$) on $\mathcal{L}_v$, this point process being the superposition of point processes with $\rho$-intensities from the planar cells adjacent to $v$. Hence it has intensity $\alpha_v$. For short we say that $v$ has $\alpha_v$ {\it corresponding} vertices in ${\mathcal Y}$.

\begin{property}
\label{property1}
Let $v$ be a vertex in ${\mathcal Y}'$. Then $v$ has $\alpha_v$ corresponding vertices in ${\mathcal Y}$ and each one
is adjacent to $m'_\mathsf{E}(v)+1$ cells and to $m'_\mathsf{E}(v)+3$ plates of ${\mathcal Y}$.
\end{property}

Furthermore the column tessellation has only horizontal and vertical edges denoted by $\mathsf{E[hor]}$ and $\mathsf{E[vert]}$, respectively. All horizontal edges are $\pi$-edges with three emanating plates.
For each edge $e$ of ${\mathcal Y}'$, we have two planar cells adjacent to this edge. When we cut the two corresponding columns by different horizontal planes, the intensity of horizontal edges of ${\mathcal Y}$ in the common face of the two neighbouring columns is $\alpha_e$, and all these edges are translations of $e$. Besides, the intensity of vertical edges of $\mathcal{Y}$ on a line $\mathcal{L}_v$ is $\alpha_v$.

\begin{property}
\label{property2}
An edge $e$ of ${\mathcal Y}'$ corresponds to $\alpha_e$ horizontal edges of ${\mathcal Y}$.\\
Any horizontal edge of ${\mathcal Y}$ is a $\pi$-edge with three emanating plates, two of them are vertical, the third one is a horizontal plate.\\
A vertex $v$ of ${\mathcal Y}'$ corresponds to  $\alpha_v$ vertical edges of ${\mathcal Y}$, where each one is adjacent to $m'_\mathsf{E}(v)$ plates of ${\mathcal Y}$.
 \end{property}
 These correspondence relations between ${\mathcal Y}'$ and ${\mathcal Y}$ will be more and more refined in due course.

\section{Relations for characteristics of a column tessellation}
\label{relations}
\subsection{Intensities of primitive elements of column tessellations} \ \\
As a first step we will consider how the intensities $\lambda_\mathsf{X}$ of the primitive elements $\mathsf{X} \in \{\mathsf{V},\mathsf{E},\mathsf{P},\mathsf{Z}\}$ of a column tessellation ${\mathcal Y}$ depend on characteristic of the planar tessellation ${\mathcal Y}'$. For those relations we need the intensities $\lambda'_\mathsf{Z}$ and $\lambda'_\mathsf{V}$, the mean  $\rho$-intensity $\bar{\rho}_\mathsf{Z}$ and the mean total $\rho$-intensity $\bar{\alpha}_\mathsf{V}$ of ${\mathcal Y}'$:

\begin{prop}
\label{prop:int}
The intensities of primitive elements of a column tessellation ${\mathcal Y}$ depend on ${\mathcal Y}'$ and the cell marks $\rho_z$ as follows
\begin{enumerate}
\item [(i)] $\lambda_\mathsf{V}=\lambda'_\mathsf{V}\bar{\alpha}_\mathsf{V}$,
\item [(ii)] $\lambda_\mathsf{E}=2\lambda'_\mathsf{V}\bar{\alpha}_\mathsf{V}$,
\item [(iii)] $\lambda_\mathsf{P}=\lambda'_\mathsf{V}\bar{\alpha}_\mathsf{V}+\lambda'_\mathsf{Z}\bar{\rho}_\mathsf{Z}$,
\item [(iv)] $\lambda_\mathsf{Z}=\lambda'_\mathsf{Z}\bar{\rho}_\mathsf{Z}$.
\end{enumerate}
\emph{For a refined partition of the classes $\mathsf{E}$ and $\mathsf{P}$ of ${\mathcal Y}$ into
horizontal and vertical elements we obtain}
\begin{enumerate}
\item[(v)] $\lambda_\mathsf{E[hor]}=\lambda_\mathsf{E[vert]}=\lambda'_\mathsf{V}\bar{\alpha}_\mathsf{V}$,
\item[(vi)] $\lambda_\mathsf{P[hor]}=\lambda'_\mathsf{Z}\bar{\rho}_\mathsf{Z}$,
\ $\lambda_\mathsf{P[vert]}=\lambda'_\mathsf{V}\bar{\alpha}_\mathsf{V}$.
\end{enumerate}\end{prop}

\begin{proof}
\noindent
With Property \ref{property1} we obtain (i).

Using Property \ref{property2} and the mean value relation $\lambda'_\mathsf{V}\bar{\alpha}_\mathsf{V}=\lambda'_\mathsf{Z}\mathbb{E}'_\mathsf{Z}(m'_\mathsf{V}(z)\rho_z)=\qquad $ $\lambda'_\mathsf{Z}\mathbb{E}'_\mathsf{Z}(m'_\mathsf{E}(z)\rho_z)=\lambda'_\mathsf{E}\bar{\alpha}_\mathsf{E}$ we have (v) and $\lambda_\mathsf{E}=\lambda_\mathsf{E[hor]}+\lambda_\mathsf{E[vert]}$ yields (ii).\\
From the construction of the column tessellation (iv) is obvious.\\
With $\lambda_\mathsf{V}-\lambda_\mathsf{E}+\lambda_\mathsf{P}-\lambda_\mathsf{Z}=0$ we obtain (iii)
and $\lambda_\mathsf{P[hor]}=\lambda_\mathsf{Z}$ leads to (vi).
\end{proof}

Further intensities can be calculated using properties of the column tessellation or relations given in Theorem \ref{thm:adja} and in \cite{WC}, for example\\
\hspace*{5mm}
\begin{tabular}{ll}
intensity of plate-sides: & $\lambda_\mathsf{P_1}=\lambda'_\mathsf{Z_0}\bar{\beta}_\mathsf{Z_0}+4\lambda'_\mathsf{V}\bar{\alpha}_\mathsf{V}$, \\
intensity of cell-facets: & $\lambda_\mathsf{Z_2}=2\lambda'_\mathsf{Z}\bar{\rho}_\mathsf{Z}+\lambda'_\mathsf{Z_0}\bar{\beta}_\mathsf{Z_0}$,\\
intensity of cell-ridges: & $\lambda_\mathsf{Z_1}=3\lambda'_\mathsf{Z_0}\bar{\beta}_\mathsf{Z_0}.$
\end{tabular}\\
Note that for the calculation of those intensities, the interior parameter $\phi$ of the planar tessellation is a necessary input, because $\lambda'_\mathsf{Z_0}$ depends on $\phi$.

\subsection{Topological mean values of column tessellations} \ \\
We now present the three adjacency parameters $\mu_{\mathsf{VE}},\ \mu_\mathsf{EP},\ \mu_\mathsf{PV}$ and the four interior parameters $\xi,\ \kappa,\ \psi,\ \tau$ of a column tessellation. To clarify their dependence on the basic planar tessellation ${\mathcal Y}'$, we need from ${\mathcal Y}'$ the mean number of emanating edges of the typical vertex $\mu'_\mathsf{VE}$, the interior parameter $\phi$ and five already--mentioned mean values $\bar{\rho}_\mathsf{Z}$, $\bar{\alpha}_\mathsf{V}$, $\bar{\alpha}_{\mathsf{V[\pi]}}$, $\bar{\beta}_{\mathsf{Z_0}}$, $\bar{\gamma}_\mathsf{V}$.

\begin{thm}
\label{thm:adja}
The seven topological mean values of a column tessellation ${\mathcal Y}$ are given by seven parameters of the underlying planar tessellation ${\mathcal Y}'$ as follows
\begin{align}
\mu_\mathsf{VE}&=4,\\
\mu_\mathsf{PV}&=\frac{2(3\bar{\alpha}_\mathsf{V}+\bar{\gamma}_\mathsf{V})}{2\bar{\alpha}_\mathsf{V}+(\mu'_\mathsf{VE}-2)\bar{\rho}_\mathsf{Z}},\\
\mu_\mathsf{EP}&=\frac{1}{2}\frac{\bar{\gamma}_\mathsf{V}}{\bar{\alpha}_\mathsf{V}}+\frac{3}{2},\\
\xi&=\frac{1}{2}\phi\frac{\bar{\alpha}_{\mathsf{V[\pi]}}}{\bar{\alpha}_\mathsf{V}}+\frac{1}{2},\\
\kappa&=\phi\frac{\bar{\alpha}_{\mathsf{V[\pi]}}}{\bar{\alpha}_\mathsf{V}}+(\mu'_\mathsf{VE}-\phi)\frac{\bar{\beta}_{\mathsf{Z_0}}}{\bar{\alpha}_\mathsf{V}}-1,\\
\psi&=\frac{\bar{\gamma}_\mathsf{V}}{\bar{\alpha}_\mathsf{V}}-\phi\frac{\bar{\alpha}_{\mathsf{V[\pi]}}}{\bar{\alpha}_\mathsf{V}}-3(\mu'_\mathsf{VE}-\phi)\frac{\bar{\beta}_{\mathsf{Z_0}}}{\bar{\alpha}_\mathsf{V}}+2,\\
\tau&=\frac{\bar{\gamma}_\mathsf{V}}{\bar{\alpha}_\mathsf{V}}-(\mu'_\mathsf{VE}-\phi)\frac{\bar{\beta}_{\mathsf{Z_0}}}{\bar{\alpha}_\mathsf{V}}-1.
\end{align}
\end{thm}

\begin{proof}
 (1) Each vertex of ${\mathcal Y}$ arises by the intersection of an infinite cylindrical column with a horizontal plane, hence the vertex has $4$ outgoing edges, $2$ of them are horizontal and the other $2$ are vertical and collinear. So we have $m_\mathsf{E}(v)=4\text{ for all }v \in \mathsf{V}$.

 \bigskip
 (2) From Property \ref{property1} we have for the mean number of plates adjacent to the typical vertex
$$\lambda_\mathsf{V}\mu_\mathsf{VP}=\lambda'_\mathsf{V}\mathbb{E}'_\mathsf{V}[\alpha_v(m'_\mathsf{E}(v)+3)]=\lambda'_\mathsf{V}\bar{\gamma}_\mathsf{V}+3\lambda'_\mathsf{V}\bar{\alpha}_\mathsf{V}.$$
With $\lambda_\mathsf{V}\mu_\mathsf{VP}=\lambda_\mathsf{P}\mu_\mathsf{PV}$ and (iii) of Proposition \ref{prop:int} we obtain (2).

 \bigskip
 (3) A column tessellation has horizontal and vertical plates
$$\lambda_\mathsf{E}\mu_\mathsf{EP}=\lambda_{\mathsf{E[hor]}}\mu_{\mathsf{E[hor]}\mathsf{P}}+\lambda_{\mathsf{E[vert]}}\mu_{\mathsf{E[vert]}\mathsf{P}}.$$
Obviously, $\mu_{\mathsf{E[hor]}\mathsf{P}}=3$ and we have $\lambda_{\mathsf{E[hor]}}=\lambda'_\mathsf{V}\bar{\alpha}_\mathsf{V}$ from (v). Each vertical edge corresponding to a vertex $v$ in the planar tessellation is adjacent to $m'_\mathsf{E}(v)$ plates; see Property \ref{property2}. Therefore
$$\lambda_{\mathsf{E[vert]}}\mu_{\mathsf{E[vert]}\mathsf{P}}=\lambda'_\mathsf{V}\mathbb{E}'_\mathsf{V}(\alpha_v m'_\mathsf{E}(v))=\lambda'_\mathsf{V}\bar{\gamma}_\mathsf{V}$$
and hence, with (ii) of Proposition \ref{prop:int},
 $$\mu_\mathsf{EP}=\frac{\lambda'_\mathsf{V}\bar{\alpha}_\mathsf{V}\cdot 3+\lambda'_\mathsf{V}\bar{\gamma}_\mathsf{V}}{2\lambda'_\mathsf{V}\bar{\alpha}_\mathsf{V}}=\frac{1}{2}\frac{\bar{\gamma}_\mathsf{V}}{\bar{\alpha}_\mathsf{V}}+\frac{3}{2}.$$

 \bigskip
 (4) To find the relations for the interior parameters we have to refine the correspondence relations between ${\mathcal Y}'$ and ${\mathcal Y}$ into two cases: whether a vertex of ${\mathcal Y}'$ is a $\pi$-vertex or not. To calculate the intensity of $\pi$-edges $\lambda_{\mathsf{E}[\pi]}$ of $\mathcal{Y}$ we note firstly that all horizontal edges are $\pi$-edges and secondly that a vertical edge is a $\pi$-edge if the corresponding vertex $v\in\mathcal{Y}'$ is a $\pi$-vertex. Hence
 $$\lambda_{\mathsf{E}[\pi]}=\lambda_{\mathsf{E[hor]}}+\lambda'_{\mathsf{V[\pi]}}\mathbb{E}'_{\mathsf{V[\pi]}}(\alpha_{v[\pi]})=\lambda'_\mathsf{V}\bar{\alpha}_\mathsf{V}+\lambda'_{\mathsf{V[\pi]}}\bar{\alpha}_{\mathsf{V[\pi]}},$$
which implies, using $\lambda_{\mathsf{E}[\pi]}= \lambda_\mathsf{E} \xi$ and $\lambda_\mathsf{E}=2\lambda'_\mathsf{V}\bar{\alpha}_\mathsf{V}$,
$$\xi=\frac{\lambda'_\mathsf{V}\bar{\alpha}_\mathsf{V}+\lambda'_{\mathsf{V[\pi]}}\bar{\alpha}_{\mathsf{V[\pi]}}}{2\lambda'_\mathsf{V}\bar{\alpha}_\mathsf{V}}=\frac{1}{2}\phi\frac{\bar{\alpha}_{\mathsf{V[\pi]}}}{\bar{\alpha}_\mathsf{V}}+\frac{1}{2}.$$

\bigskip
(5) To prove the next three relations we again have to refine the corresponding relations between ${\mathcal Y}'$ and ${\mathcal Y}$.  We consider when the vertices of a column tessellation are hemi-vertices or not.
If the vertex $v$ of ${\mathcal Y}'$ is not a $\pi$-vertex, then all $\alpha_v$ corresponding vertices of ${\mathcal Y}$ are not hemi-vertices. If the vertex is a $\pi$-vertex, denoted by $v[\pi]$, then $\beta_{v[\pi]}$ of the corresponding vertices are non-hemi-vertices, the others being hemi-vertices. Hence the intensity of hemi-vertices $\lambda_{\mathsf{V}[\kappa]}$  $=\lambda_{\mathsf{V}} \kappa$  is
\begin{align*}
\lambda_{\mathsf{V}[\kappa]}&=\lambda'_{\mathsf{V[\pi]}}\mathbb{E}'_{\mathsf{V[\pi]}}({\alpha}_{v[\pi]}-\beta_{v[\pi]})\\
&=\lambda'_{\mathsf{V[\pi]}}\bar{\alpha}_{\mathsf{V[\pi]}}-\lambda'_{\mathsf{V[\pi]}}\bar{\beta}_{\mathsf{V[\pi]}}\\
&=\lambda'_{\mathsf{V[\pi]}}\bar{\alpha}_{\mathsf{V[\pi]}}-\lambda'_\mathsf{V}\bar{\alpha}_\mathsf{V}+\lambda'_{\mathsf{Z_0}}\bar{\beta}_{\mathsf{Z_0}}
\end{align*}
using $\lambda'_\mathsf{V[\pi]}\bar{\beta}_\mathsf{V[\pi]}=\lambda'_\mathsf{V}\bar{\alpha}_\mathsf{V}-\lambda'_{\mathsf{Z_0}}\bar{\beta}_{\mathsf{Z_0}}$ from Remark \ref{rem:second-order}.
Therefore with $\lambda'_{\mathsf{Z}_0}=\lambda'_\mathsf{V}(\mu'_\mathsf{VE}-\phi)$ and Proposition \ref{prop:int}, (1)
\begin{align*}
\kappa&=\frac{\lambda'_{\mathsf{V[\pi]}}\bar{\alpha}_{\mathsf{V[\pi]}}-\lambda'_\mathsf{V}\bar{\alpha}_\mathsf{V}+\lambda'_\mathsf{V}(\mu'_\mathsf{VE}-\phi)\bar{\beta}_{\mathsf{Z_0}}}{\lambda'_\mathsf{V}\bar{\alpha}_\mathsf{V}}\\
&=\phi\frac{\bar{\alpha}_{\mathsf{V[\pi]}}}{\bar{\alpha}_\mathsf{V}}+(\mu'_\mathsf{VE}-\phi)\frac{\bar{\beta}_{\mathsf{Z_0}}}{\bar{\alpha}_\mathsf{V}}-1.
\end{align*}

\bigskip
(6) To present the parameter $\psi$, we have to find out the number of ridge-interiors adjacent to a vertex in different cases.
If the vertex $v$ of ${\mathcal Y}'$ is not a $\pi$-vertex, denoted by $v[\bar{\pi}]$, then each of the $\alpha_{v[\bar{\pi}]}$ corresponding vertices of ${\mathcal Y}$ is adjacent to $m'_\mathsf{E}(v[\bar{\pi}])-1$ ridge-interiors. If $v$ of ${\mathcal Y}'$ is a $\pi$-vertex $v[\pi]$, each of the corresponding non-hemi-vertices of ${\mathcal Y}$ is adjacent to $m'_{\mathsf{E}}(v[\pi])+1$ ridge-interiors, and each of the remaining corresponding hemi-vertices is adjacent to $m'_{\mathsf{E}}(v[\pi])-2$ ridge-interiors. Hence
\begin{align*}
\lambda_\mathsf{V}\psi&=\lambda'_\mathsf{V[\bar{\pi}]}\mathbb{E'}_\mathsf{V[\bar{\pi}]}[\alpha_{v[\bar{\pi}]}(m'_\mathsf{E}(v[\bar{\pi}])-1)] +
\lambda'_\mathsf{V[\pi]}\mathbb{E'}_\mathsf{V[\pi]}[\beta_{v[\pi]}(m'_\mathsf{E}(v[\pi])+1)]\\
& \ \ \ +
\lambda'_\mathsf{V[\pi]}\mathbb{E'}_\mathsf{V[\pi]}[(\alpha_{v[\pi]}-\beta_{v[\pi]})(m'_\mathsf{E}(v[\pi])-2)] \\
&=\lambda'_\mathsf{V}\mathbb{E'}_\mathsf{V}(\alpha_v m'_\mathsf{E}(v))-2\lambda'_\mathsf{V}\mathbb{E'}_\mathsf{V}(\alpha_v)+\lambda'_\mathsf{V[\bar{\pi}]}\mathbb{E'}_\mathsf{V[\bar{\pi}]}(\alpha_{v[\bar{\pi}]})+3\lambda'_\mathsf{V[\pi]}\mathbb{E'}_\mathsf{V[\pi]}(\beta_{v[\pi]})\\
&=\lambda'_\mathsf{V}\bar{\gamma}_\mathsf{V}-2\lambda'_\mathsf{V}\bar{\alpha}_\mathsf{V}+\lambda'_\mathsf{V}\bar{\alpha}_\mathsf{V}-\lambda'_{\mathsf{V[\pi]}}\bar{\alpha}_{\mathsf{V[\pi]}}+3\lambda'_\mathsf{V}\bar{\alpha}_\mathsf{V}-3\lambda'_{\mathsf{Z_0}}\bar{\beta}_{\mathsf{Z_0}}\\
&=\lambda'_\mathsf{V}\bar{\gamma}_\mathsf{V}-\lambda'_{\mathsf{V[\pi]}}\bar{\alpha}_{\mathsf{V[\pi]}}-3\lambda'_{\mathsf{Z_0}}\bar{\beta}_{\mathsf{Z_0}}+2\lambda'_\mathsf{V}\bar{\alpha}_\mathsf{V}.
\end{align*}
Therefore,
\begin{align*}
\psi&=\frac{\lambda'_\mathsf{V}\bar{\gamma}_\mathsf{V}-\lambda'_{\mathsf{V[\pi]}}\bar{\alpha}_{\mathsf{V[\pi]}}-3\lambda'_\mathsf{V}(\mu'_\mathsf{VE}-\phi)\bar{\beta}_{\mathsf{Z_0}}+2\lambda'_\mathsf{V}\bar{\alpha}_\mathsf{V}}{\lambda'_\mathsf{V}\bar{\alpha}_\mathsf{V}}\\
&=\frac{\bar{\gamma}_\mathsf{V}}{\bar{\alpha}_\mathsf{V}}-\phi\frac{\bar{\alpha}_{\mathsf{V[\pi]}}}{\bar{\alpha}_\mathsf{V}}-3(\mu'_\mathsf{VE}-\phi)\frac{\bar{\beta}_{\mathsf{Z_0}}}{\bar{\alpha}_\mathsf{V}}+2.
\end{align*}

\bigskip
(7) Now for the last identity we consider how the number of plate-side-interiors adjacent to a vertex of ${\mathcal Y}$ depends on the type of the corresponding vertex of ${\mathcal Y}'$: If the vertex $v$ of ${\mathcal Y}'$ is a $v[\bar{\pi}]$, then each of the corresponding $\alpha_{v[\bar{\pi}]}$ vertices of ${\mathcal Y}$ is adjacent to $m'_\mathsf{E}(v[\bar{\pi}])-2$ plate-side-interiors. If $v$ is a $v[\pi]$, each of the corresponding non-hemi-vertices is adjacent to $m'_{\mathsf{E}}(v[\pi])-1$ plate-side-interiors, and each of the other corresponding hemi-vertices is adjacent to $m'_{\mathsf{E}}(v[\pi])-2$ plate-side-interiors. Hence
\begin{align*}
\lambda_\mathsf{V}\tau&=\lambda'_\mathsf{V[\bar{\pi}]}\mathbb{E'}_\mathsf{V[\bar{\pi}]}[\alpha_{v[\bar{\pi}]}(m'_\mathsf{E}(v[\bar{\pi}])-2)] +
\lambda'_\mathsf{V[\pi]}\mathbb{E'}_\mathsf{V[\pi]}[\beta_{v[\pi]}(m'_\mathsf{E}(v[\pi])-1)]\\
& \ \ \ +
\lambda'_\mathsf{V[\pi]}\mathbb{E'}_\mathsf{V[\pi]}[(\alpha_{v[\pi]}-\beta_{v[\pi]})(m'_\mathsf{E}(v[\pi])-2)] \\
&=\lambda'_\mathsf{V}\mathbb{E'}_\mathsf{V}(\alpha_vm'_\mathsf{E}(v))-2\lambda'_\mathsf{V}\mathbb{E'}_\mathsf{V}(\alpha_v)+\lambda'_\mathsf{V[\pi]}\mathbb{E'}_\mathsf{V[\pi]}(\beta_{v[\pi]})\\
&=\lambda'_\mathsf{V}\bar{\gamma}_\mathsf{V}-2\lambda'_\mathsf{V}\bar{\alpha}_\mathsf{V}+\lambda'_\mathsf{V}\bar{\alpha}_\mathsf{V}-\lambda'_{\mathsf{Z_0}}\bar{\beta}_{\mathsf{Z_0}}\\
&=\lambda'_\mathsf{V}\bar{\gamma}_\mathsf{V}-\lambda'_\mathsf{V}\bar{\alpha}_\mathsf{V}-\lambda'_{\mathsf{Z_0}}\bar{\beta}_{\mathsf{Z_0}}.
\end{align*}
Therefore
\begin{align*}
\tau&=\frac{\lambda'_\mathsf{V}\bar{\gamma}_\mathsf{V}-\lambda'_\mathsf{V}(\mu'_\mathsf{VE}-\phi)\bar{\beta}_{\mathsf{Z_0}}-\lambda'_\mathsf{V}\bar{\alpha}_\mathsf{V}}{\lambda'_\mathsf{V}\bar{\alpha}_\mathsf{V}}\\
&=\frac{\bar{\gamma}_\mathsf{V}}{\bar{\alpha}_\mathsf{V}}-(\mu'_\mathsf{VE}-\phi)\frac{\bar{\beta}_{\mathsf{Z_0}}}{\bar{\alpha}_\mathsf{V}}-1.
\end{align*}
\end{proof}

Using mean value relations in \cite{WC}, further topological parameters can be computed. For example the mean number of vertices and edges, respectively, of the typical cell are
$$ \mu_\mathsf{ZV} = 2\frac{\bar{\gamma}_\mathsf{V}+\bar{\alpha}_\mathsf{V}}{(\mu'_\mathsf{VE}-2)\bar{\rho}_\mathsf{Z}} \quad {\rm and}  \quad \mu_\mathsf{ZE} =  2\frac{\bar{\gamma}_\mathsf{V}+3\bar{\alpha}_\mathsf{V}}{(\mu'_\mathsf{VE}-2)\bar{\rho}_\mathsf{Z}} , $$
whereas the mean number of $0$-faces and $1$-faces of the typical cell are
$$\nu_0(\mathsf{Z})=4 \bar{\beta}_{\mathsf{Z_0}} \frac{\mu'_\mathsf{VE}-\phi}{(\mu'_\mathsf{VE}-2)\bar{\rho}_\mathsf{Z}}\quad {\rm and}  \quad
\nu_1(\mathsf{Z})=6 \bar{\beta}_{\mathsf{Z_0}} \frac{\mu'_\mathsf{VE}-\phi}{(\mu'_\mathsf{VE}-2)\bar{\rho}_\mathsf{Z}}.$$

\begin{rem} \label{rem:int1-top-mv}
To calculate the intensities and topological parameters of a column tessellation with height 1 from the planar tessellation, five parameters are needed,
$$ \lambda'_{\mathsf{V}}, \quad \mu'_{\mathsf{VE}}, \quad \phi, \quad \mu'_{\mathsf{EV}[\pi]} \quad \mathrm{and} \quad \mu'^{(2)}_{\mathsf{VE}}. $$
Using Remark \ref{rem:intensity1}, Proposition \ref{prop:int} and Theorem \ref{thm:adja} and the mean value relation $\mu'_{\mathsf{V[\pi]E}}= \displaystyle{\frac{\mu'_{\mathsf{VE}}}{2 \phi}\mu'_{\mathsf{EV}[\pi]}}$  the intensities of a column tessellation with height 1 are
 $$\lambda_{\mathsf{V}}=\lambda'_{\mathsf{V}}\mu'_{\mathsf{VE}}, \quad
 \lambda_{\mathsf{E}}=2\lambda'_{\mathsf{V }}\mu'_{\mathsf{VE}}, \quad
 \lambda_{\mathsf{P}}=\frac{1}{2}\lambda'_{\mathsf{V}}(3\mu'_{\mathsf{VE}}-2), \quad
 \lambda_{\mathsf{Z}}=\frac{1}{2}\lambda'_{\mathsf{V}}(\mu'_{\mathsf{VE}}-2), $$
the cyclic adjacency parameters are
$$
\mu_{\mathsf{VE}}=4, \quad
\mu_{\mathsf{PV}}=\frac{2}{3\mu'_{\mathsf{VE}}-2}(3\mu'_{\mathsf{VE}}+\mu'^{(2)}_{\mathsf{VE}}), \quad
\mu_{\mathsf{EP}}=\frac{1}{2\mu'_{\mathsf{VE}}}(3\mu'_{\mathsf{VE}}+\mu'^{(2)}_{\mathsf{VE}}), $$
and for the interior parameters we obtain
\begin{alignat*}{2} \xi=&\ \frac{1}{2}+\frac{1}{4}\mu'_{\mathsf{EV}[\pi]},& \quad
\kappa=&\ \frac{1}{2}\mu'_{\mathsf{EV}[\pi]}-\frac{\phi}{\mu'_{\mathsf{VE}}}, \\
\psi=&\ \frac{\mu'^{(2)}_{\mathsf{VE}}+3\phi}{\mu'_{\mathsf{VE}}}-1-\frac{1}{2}\mu'_{\mathsf{EV}[\pi]}, \quad&
\tau=&\ \frac{\mu'^{(2)}_{\mathsf{VE}}+\phi}{\mu'_{\mathsf{VE}}}-2.
\end{alignat*}
\end{rem}

\begin{rem} In \cite{CW} constraints on the topological parameters of spatial tessellations are considered.
Because the second moment $\mu'^ {(2)}_\mathsf{VE}$ of a planar tessellation is unbounded, see \cite{CW}, in the class of column tessellations with height 1 the mean values $\mu_\mathsf{EP}, \ \mu_\mathsf{PV}$ and $ \tau, \ \psi$ are unbounded. Further constraints are as follows.
\begin{prop}
\label{prop:constraints}
The constraints for the topological mean values of a column tessellation ${\mathcal Y}$ with height 1 depending on $\mu'_{\mathsf{VE}}$ and $\phi$ of  ${\mathcal Y}'$ are as follows\\[2mm]
$$\begin{array}{rcccccccl}
\frac{36}{7} & \leq & \displaystyle{\frac{2\mu'_{\mathsf{VE}}(3+\mu'_{\mathsf{VE}})}{3\mu'_{\mathsf{VE}}-2}} & \leq & \mu_{\mathsf{PV}},  \\[3mm]
3 &\leq& \frac{1}{2} (3 + \mu'_{\mathsf{VE}}) &\leq & \mu_{\mathsf{EP}},\\[3mm]
\frac{1}{2} &\leq& \displaystyle{\frac{1}{2} + \frac{3}{2} \frac{\phi}{\mu'_{\mathsf{VE}}}} &\leq & \xi & \leq  &
\displaystyle{1 - \frac{3(1- \phi)}{2 \mu'_{\mathsf{VE}}}} & \leq & 1,\\[3mm]
0 &\leq & \displaystyle{\frac{2 \phi}{\mu'_{\mathsf{VE}}}} &\leq & \kappa & \leq  & \displaystyle{ 1 - \frac{3 - 2 \phi}{\mu'_{\mathsf{VE}}}}& \leq & \frac{3}{4},
 \\[3mm]
2 & \leq & \displaystyle{\mu'_{\mathsf{VE}} + \frac{3}{\mu'_{\mathsf{VE}}} - 2}& \leq & \psi, \\[3mm]
1 &\leq& \displaystyle{\mu'_{\mathsf{VE}} + \frac{\phi}{\mu'_{\mathsf{VE}}}-2} &\leq& \tau.
\end{array}$$
\end{prop}
\begin{proof}
For any planar tessellation we have
\begin{equation*}
 0 \leq \phi \leq 1 \quad {\rm and} \quad 3 \leq \mu'_{\mathsf{VE}} \leq 6 - 2 \phi,
 \end{equation*}
as shown in \cite{WC}.
Furthermore it is evident that $3 \leq \mu'_{\mathsf{V}[\pi]\mathsf{E}}$ and $3 \leq \mu'_{\mathsf{V}[\bar{\pi}]\mathsf{E}}$. With $\mu'_{\mathsf{VE}}=\phi\mu'_{\mathsf{V}[\pi]\mathsf{E}}+(1-\phi)\mu'_{\mathsf{V}[\bar{\pi}]\mathsf{E}}$ we obtain the following constraints for the mean number of emanating edges of the typical $\pi$-vertex
$$ 3 \leq  \mu'_{\mathsf{V}[\pi]\mathsf{E}} \leq  \displaystyle{\frac{\mu'_{\mathsf{VE}}}{\phi } - \frac{3(1-  \phi)}{\phi}}.$$
Hence the constraints for ${\mu'_{\mathsf{EV}[\pi]}}$ are
\begin{equation*}
\displaystyle{\frac{6 \phi}{\mu'_{\mathsf{VE}}}} \leq \mu'_{\mathsf{EV[\pi]}} \leq 2 - \displaystyle{\frac{6(1-\phi)}{\mu'_{\mathsf{VE}}}}
\end{equation*}
using $ \mu'_{\mathsf{EV}[\pi]} = \frac{2 \phi}{\mu'_{\mathsf{VE}}} \mu'_{\mathsf{V}[\pi]\mathsf{E}} $.

Applying these results to Remark \ref{rem:int1-top-mv} leads to the constraints for column tessellations with height 1.
\end{proof}
\end{rem}
\subsection{Metric mean values of column tessellations}
\subsubsection{Notations and mean values corresponding to the length measure}

Firstly we consider mean values corresponding to the length measure for the object classes  $\mathsf{X}\in\{\mathsf{E},\mathsf{P},\mathsf{Z}\}$  in $\mathcal Y$, those denoted by\\[2mm]
\begin{tabular}{lcl}
$\bar{\ell}_\mathsf{X}$ & -- & the mean total length of all $1$-faces of the typical $\mathsf{X}$-object, where\\&&$\dim$($\mathsf{X}$-object)$\geq 1$.
\end{tabular}\\[2mm]
This means for special object classes\\[2mm]
\hspace*{5mm}
\begin{tabular}{lcl}
$\bar{\ell}_\mathsf{E}$ & -- & the mean length of the typical edge,\\
$\bar{\ell}_\mathsf{P}$ & -- & the mean perimeter of the typical plate and\\
$\bar{\ell}_\mathsf{Z}$ & -- & the mean total length of all ridges of the typical cell.
\end{tabular}\\[2mm]
We can also define $\bar{\ell}_{\mathsf{X_k}}$ and $\bar{\ell}_{\mathsf{X[.]}}$ in a similar way. For example\\[2mm]
\hspace*{5mm}
\begin{tabular}{lcl}
$\bar{\ell}_{\mathsf{E[hor]}}$, $\bar{\ell}_{\mathsf{E[vert]}}$ and $\bar{\ell}_{\mathsf{E[\pi]}}$ & -- & the mean length of the typical horizontal edge, \\&&the typical vertical edge and
 the typical $\pi$-edge\\&& in that order and
\end{tabular}\\
\hspace*{5mm}
\begin{tabular}{lcl}
$\bar{\ell}_{\mathsf{P_1}}$, $\bar{\ell}_{\mathsf{Z_1}}$ & -- & the mean length of the typical plate-side and the typical ridge,  \\ &&respectively, \\
$\bar{\ell}_{\mathsf{Z_2}}$ & -- & the mean perimeter  of the typical facet.
\end{tabular}\\[2mm]
These notations do not include for instance the mean total length of all edges of the typical cell. Therefore we use again the adjacency concept, analog to the topological mean values $\mu_{\mathsf{XY}}$:  \\[2mm]
\hspace*{5mm}
\begin{tabular}{lcl}
$\bar{\ell}_{\mathsf{XY}}$ & -- & the mean total length of all $\mathsf{Y}$-objects adjacent to the typical \\&&$\mathsf{X}$-object, where $\dim$($\mathsf{Y}$-object)$=1$.
\end{tabular}\\[2mm]
For $\mathsf{X} = \mathsf{Z}$ and $\mathsf{Y} = \mathsf{E}$ we have\\[2mm]
\hspace*{5mm}
\begin{tabular}{lcl}
$\bar{\ell}_\mathsf{ZE}$ & -- & the mean total length of all edges adjacent to the typical cell.
\end{tabular}

Some of these $\bar{\ell}_{\mathsf{XY}}$ mean values can be easily  determined, for example
$$ \bar{\ell}_\mathsf{PE}=\bar{\ell}_\mathsf{P},\quad
\bar{\ell}_\mathsf{Z_1E}=\bar{\ell}_\mathsf{Z_1}, \quad
\bar{\ell}_\mathsf{P_1E}=\bar{\ell}_\mathsf{P_1}, $$
but other examples (see Proposition \ref{prop:lemeva2}) are more complicated and demonstrate the necessity of the notation.

Using the parameter $\bar{\gamma}_\mathsf{E}$ of the planar tessellation $\mathcal{Y}'$ - the length-weighted total $\rho$-intensity of the cells adjacent to the typical edge, we can calculate the mean values corresponding to the length measure of a column tessellation.

\begin{thm}
\label{thm:lemeva}
Three mean values of primitive elements corresponding to the length measure of the column tessellation are given as follows:
\begin{align}
\bar{\ell}_\mathsf{E}&=\frac{1}{2}\bigg(\frac{\bar{\gamma}_\mathsf{E}}{\bar{\alpha}_\mathsf{E}}+\frac{1}{\bar{\alpha}_\mathsf{V}}\bigg);\\
\bar{\ell}_\mathsf{P}&=\frac{(3\bar{\gamma}_\mathsf{E}+2)\mu'_\mathsf{VE}}{(\mu'_\mathsf{VE}-2)\bar{\rho}_\mathsf{Z}+\mu'_\mathsf{VE}\bar{\alpha}_\mathsf{E}};\\
\bar{\ell}_\mathsf{Z}&=\frac{2(\mu'_\mathsf{VE}\bar{\gamma}_\mathsf{E}+\mu'_\mathsf{VE}-\phi)}{(\mu'_\mathsf{VE}-2)\bar{\rho}_\mathsf{Z}}.
\end{align}
\end{thm}

\begin{proof}
 (8) Recalling that a column tessellation has only horizontal and vertical edges
$$\lambda_\mathsf{E}\bar{\ell}_\mathsf{E}=\lambda_{\mathsf{E[hor]}}\bar{\ell}_{\mathsf{E[hor]}}+\lambda_{\mathsf{E[vert]}}\bar{\ell}_{\mathsf{E[vert]}}=\lambda'_\mathsf{E}\bar{\gamma}_\mathsf{E}+\lambda'_\mathsf{V}.$$
With (ii) from Proposition \ref{prop:int} and the relation $\lambda'_\mathsf{E}\bar{\alpha}_\mathsf{E}=\lambda'_\mathsf{V}\bar{\alpha}_\mathsf{V}$ we obtain (8).

 (9) Similarly, for the plates of $\mathcal{Y}$ we have
$$\lambda_\mathsf{P}\bar{\ell}_\mathsf{P}=\lambda_{\mathsf{P[hor]}}\bar{\ell}_{\mathsf{P[hor]}}+\lambda_{\mathsf{P[vert]}}\bar{\ell}_{\mathsf{P[vert]}}.$$
It is not difficult to see that $\lambda_{\mathsf{P[hor]}}\bar{\ell}_{\mathsf{P[hor]}}=\lambda'_\mathsf{Z}\mathbb{E}'_\mathsf{Z}(\ell'(z)\rho_z)=\lambda'_\mathsf{E}\bar{\gamma}_\mathsf{E}$ and $\lambda_{\mathsf{P[vert]}}\bar{\ell}_{\mathsf{P[vert]}}=2\lambda_{\mathsf{E[hor]}}\bar{\ell}_{\mathsf{E[hor]}}+2\lambda'_\mathsf{E}=2\lambda'_\mathsf{E}\bar{\gamma}_\mathsf{E}+2\lambda'_\mathsf{E}$, then
$$\lambda_\mathsf{P}\bar{\ell}_\mathsf{P}=3\lambda'_\mathsf{E}\bar{\gamma}_\mathsf{E}+2\lambda'_\mathsf{E},$$
which implies (9).

 (10) To determine the mean total length of the typical ridge we use the following
$$\lambda_\mathsf{Z}\bar{\ell}_\mathsf{Z}=2\lambda_{\mathsf{P[hor]}}\bar{\ell}_{\mathsf{P[hor]}}+\lambda'_\mathsf{Z_0},$$
and we get (10).
\end{proof}

Other mean values corresponding to the length measure of the column tessellation can be computed in the same way by separating the roles of horizontal objects and vertical objects. For example,
\begin{align*}
&\text{the mean length of the typical $\pi$-edge }\bar{\ell}_{\mathsf{E[\pi]}}=\frac{\mu'_\mathsf{VE}\bar{\gamma}_\mathsf{E}+2\phi}{2(\bar{\alpha}_\mathsf{V}+\phi\bar{\alpha}_\mathsf{V[\pi]})},\\
&\text{the mean length of the typical ridge } \bar{\ell}_\mathsf{Z_1}=\frac{1}{3\bar{\beta}_\mathsf{Z_0}}\bigg(\frac{\mu'_\mathsf{VE}\bar{\gamma}_\mathsf{E}}{\mu'_\mathsf{VE}-\phi}+1\bigg),\\
&\text{the mean length of the typical plate-side } \bar{\ell}_\mathsf{P_1}=\frac{\mu'_\mathsf{VE}(3\bar{\gamma}_\mathsf{E}+2)}{2(\mu'_\mathsf{VE}-\phi)+4\mu'_\mathsf{VE}\bar{\alpha}_\mathsf{E}},\\
&\text{the mean perimeter of the typical facet } \bar{\ell}_\mathsf{Z_2}=\frac{2\mu'_\mathsf{VE}\bar{\gamma}_\mathsf{E}+2(\mu'_\mathsf{VE}-\phi)}{(\mu'_\mathsf{VE}-2)\bar{\rho}_\mathsf{Z}+(\mu'_\mathsf{VE}-\phi)\bar{\beta}_\mathsf{Z_0}}.
\end{align*}

We also take care of results for some $\bar{\ell}_{\mathsf{XY}}$. It is interesting for us to calculate\\[2mm]
\hspace*{5mm}
\begin{tabular}{lcl}
$\bar{\ell}_\mathsf{ZE}$ & -- & the mean total length of all edges adjacent to the typical cell and\\
$\bar{\ell}_\mathsf{Z_2E}$ & -- & the mean total length of all edges adjacent to the typical facet.
\end{tabular}\\[2mm]

\begin{prop}
\label{prop:lemeva2}
The values of $\bar{\ell}_\mathsf{ZE}$ and $\bar{\ell}_\mathsf{Z_2E}$ are given as follows:
\begin{align*}
\bar{\ell}_\mathsf{ZE}&=\frac{\mu'_\mathsf{VE}(3\bar{\gamma}_\mathsf{E}+2)}{(\mu'_\mathsf{VE}-2)\bar{\rho}_\mathsf{Z}};\\
\bar{\ell}_\mathsf{Z_2E}&=\frac{5\mu'_\mathsf{VE}\bar{\gamma}_\mathsf{E}+4\mu'_\mathsf{VE}-2\phi}{2[(\mu'_\mathsf{VE}-2)\bar{\rho}_\mathsf{Z}+(\mu'_\mathsf{VE}-\phi)\bar{\beta}_\mathsf{Z_0}]}.
\end{align*}
\end{prop}

\begin{proof}
Based on the properties of the column tessellation, it is not difficult to see that
$$\lambda_\mathsf{Z}\bar{\ell}_\mathsf{ZE}=\lambda_\mathsf{Z}\bar{\ell}_\mathsf{ZE[hor]}+\lambda_\mathsf{Z}\bar{\ell}_\mathsf{ZE[vert]}=3\lambda_{\mathsf{E[hor]}}\bar{\ell}_{\mathsf{E[hor]}}+\lambda'_\mathsf{Z}\mu'_\mathsf{ZV}=3\lambda'_\mathsf{E}\bar{\gamma}_\mathsf{E}+2\lambda'_\mathsf{E}$$
and
\begin{align*}
\lambda_\mathsf{Z_2}\bar{\ell}_\mathsf{Z_2E}&=\lambda_\mathsf{Z_2[hor]}\bar{\ell}_\mathsf{Z_2[hor]E}+\lambda_\mathsf{Z_2[vert]}\bar{\ell}_\mathsf{Z_2[vert]E}=2\lambda_\mathsf{P[hor]}\bar{\ell}_\mathsf{P[hor]}+\lambda_\mathsf{Z}\bar{\ell}_\mathsf{ZE}+\lambda'_\mathsf{Z_0}\\
&=2\lambda'_\mathsf{E}\bar{\gamma}_\mathsf{E}+3\lambda'_\mathsf{E}\bar{\gamma}_\mathsf{E}+2\lambda'_\mathsf{E}+\lambda'_\mathsf{Z_0}\\
&=5\lambda'_\mathsf{E}\bar{\gamma}_\mathsf{E}+\lambda'_\mathsf{Z}\frac{4\mu'_\mathsf{VE}-2\phi}{\mu'_\mathsf{VE}-2},
\end{align*}
which completes our proof.
\end{proof}

For the mean values corresponding to the length measure of column tessellations of constant cell-height $1$, we have (using the metric parameter $ \bar{\ell}'_\mathsf{E}$ from the planar tessellation)
\begin{alignat*}{3}
\bar{\ell}_\mathsf{E}=&\frac{1}{2}\bigg(\bar{\ell}'_\mathsf{E}+\frac{1}{\mu'_\mathsf{VE}}\bigg),&
    \quad\bar{\ell}_\mathsf{P}=&\frac{2\mu'_\mathsf{VE}(3\bar{\ell}'_\mathsf{E}+1)}{3\mu'_\mathsf{VE}-2},&
    \quad\bar{\ell}_\mathsf{Z}=&\frac{2(2\mu'_\mathsf{VE}\bar{\ell}'_\mathsf{E}+\mu'_\mathsf{VE}-\phi)}{\mu'_\mathsf{VE}-2},\\
\bar{\ell}_{\mathsf{E[\pi]}}=&\frac{2(\mu'_\mathsf{VE}\bar{\ell}'_\mathsf{E}+\phi)}{\mu'_\mathsf{VE}(2+\mu'_{\mathsf{EV[\pi]}})},&
    \quad\bar{\ell}_\mathsf{Z_1}=&\frac{1}{3}+\frac{2\mu'_\mathsf{VE}\bar{\ell}'_\mathsf{E}}{3(\mu'_\mathsf{VE}-\phi)},&
    \quad\bar{\ell}_\mathsf{P_1}=&\frac{\mu'_\mathsf{VE}(3\bar{\ell}'_\mathsf{E}+1)}{5\mu'_\mathsf{VE}-\phi},\\
\bar{\ell}_\mathsf{Z_2}=&\frac{2(2\mu'_\mathsf{VE}\bar{\ell}'_\mathsf{E}+\mu'_\mathsf{VE}-\phi)}{2\mu'_\mathsf{VE}-\phi-2},&
    \hspace{3mm} \bar{\ell}_\mathsf{ZE}=&\frac{2\mu'_\mathsf{VE}(3\bar{\ell}'_\mathsf{E}+1)}{\mu'_\mathsf{VE}-2},&
    \bar{\ell}_\mathsf{Z_2E}=&\frac{5\mu'_\mathsf{VE}\bar{\ell}'_\mathsf{E}+2\mu'_\mathsf{VE}-\phi}{2\mu'_\mathsf{VE}-\phi-2}.
\end{alignat*}

\subsubsection{Notations and mean values corresponding to the area measure}
An analogue notation is used for mean values corresponding to the area measure \\
\begin{tabular}{lcl}
$\bar{a}_\mathsf{X}$ & -- & the mean total area of all $2$-faces of the typical $\mathsf{X}$-object, where \\&&$\dim$($\mathsf{X}$-object)$\geq 2$.
\end{tabular}\\[2mm]
In this case we have\\[2mm]
\hspace*{5mm}
\begin{tabular}{lcl}
$\bar{a}_\mathsf{P}$ & -- & the mean area of the typical plate, \\
$\bar{a}_\mathsf{Z}$ & -- & the mean surface area of the typical cell and \\
$\bar{a}_{\mathsf{Z_2}}$ & -- & the mean area of the typical cell-facet.
\end{tabular}

To determine mean values corresponding to the area measure of the forms $\bar{a}_\mathsf{X}$, $\bar{a}_\mathsf{X_k}$ and $\bar{a}_\mathsf{X[.]}$ of the column tessellation $\mathcal{Y}$, we use two additional mean values of the planar tessellation $\mathcal{Y}'$, namely $\bar{\gamma}_\mathsf{Z}$ and $\bar{\ell}'_\mathsf{E}$.

\begin{thm}
The mean area of the typical plate of the column tessellation is given as follows:
\item $$\bar{a}_\mathsf{P}=\frac{(\mu'_\mathsf{VE}-2)\bar{\gamma}_\mathsf{Z}+\mu'_\mathsf{VE}\bar{\ell}'_\mathsf{E}}{(\mu'_\mathsf{VE}-2)\bar{\rho}_\mathsf{Z}+\mu'_\mathsf{VE}\bar{\alpha}_\mathsf{E}}.$$
\end{thm}

\begin{proof}
Because the column tessellation has only horizontal and vertical plates, we have
$$\lambda_\mathsf{P}\bar{a}_\mathsf{P}=\lambda_\mathsf{P[hor]}\bar{a}_\mathsf{P[hor]}+\lambda_\mathsf{P[vert]}\bar{a}_\mathsf{P[vert]}=\lambda'_\mathsf{Z}\bar{\gamma}_\mathsf{Z}+\lambda'_\mathsf{E}\bar{\ell}'_\mathsf{E},$$
which implies our result.
\end{proof}

Using the fact that each plate of a spatial tessellation always belongs to two cells and two facets, we have $\lambda_\mathsf{Z}\bar{a}_\mathsf{Z}=2\lambda_\mathsf{P}\bar{a}_\mathsf{P}$ and $\lambda_\mathsf{Z_2}\bar{a}_\mathsf{Z_2}=2\lambda_\mathsf{P}\bar{a}_\mathsf{P}$. With the help of Proposition \ref{prop:int} we infer that
\begin{align*}
\bar{a}_\mathsf{Z}&=\frac{2}{\bar{\rho}_\mathsf{Z}}\bigg(\bar{\gamma}_\mathsf{Z}+\frac{\mu'_\mathsf{VE}\bar{\ell}'_\mathsf{E}}{\mu'_\mathsf{VE}-2}\bigg),\\
\bar{a}_\mathsf{Z_2}&=\frac{(\mu'_\mathsf{VE}-2)\bar{\gamma}_\mathsf{Z}+\mu'_\mathsf{VE}\bar{\ell}'_\mathsf{E}}{(\mu'_\mathsf{VE}-2)\bar{\rho}_\mathsf{Z}+(\mu'_\mathsf{VE}-\phi)\bar{\beta}_\mathsf{Z_0}}.
\end{align*}

Also for the area measure we can consider mean values of type $\bar{a}_{\mathsf{XY}}$ - the mean total area of all $\mathsf{Y}$-type objects adjacent to the typical $\mathsf{X}$-object.
Again some relations are obvious:
\begin{itemize}
	\item $\bar{a}_\mathsf{ZP}=\bar{a}_\mathsf{Z}$;
	\item $\bar{a}_\mathsf{Z_2P}=\bar{a}_\mathsf{Z_2}.$
\end{itemize}
But an interesting value of this  $\bar{a}_{\mathsf{XY}}$--type is the mean total area of all facets adjacent to the typical cell, namely $\bar{a}_\mathsf{ZZ_2}$. In a facet-to-facet spatial tessellation, it is easy to see that $\bar{a}_\mathsf{ZZ_2}=2\bar{a}_\mathsf{Z}$. Because each cell-facet is a plate and the class $\mathsf{Z_2}$ of cell-facets is equal to the class $\mathsf{P}$ of plates up to the multiplicity $2$. It is difficult to determine $\bar{a}_\mathsf{ZZ_2}$ for an arbitrary non-facet-to-facet spatial tessellation, we only know that $\bar{a}_\mathsf{ZZ_2}\geq \bar{a}_\mathsf{Z}$. But for a column tessellation we can compute $\bar{a}_\mathsf{ZZ_2}$, using the fact that each horizontal facet of a cell is also a facet of one another cell and each vertical facet is an element of $\mathsf{Z_2}$ with multiplicity 1. Therefore we obtain
$$\lambda_\mathsf{Z}\bar{a}_\mathsf{ZZ_2}=\lambda_\mathsf{Z}\bar{a}_\mathsf{Z}+2\lambda_\mathsf{P[hor]}\bar{a}_\mathsf{P[hor]}=\lambda'_\mathsf{Z}\bigg(4\bar{\gamma}_\mathsf{Z}+2\frac{\mu'_\mathsf{VE}\bar{\ell}'_\mathsf{E}}{\mu'_\mathsf{VE}-2}\bigg);$$
hence
$$\bar{a}_\mathsf{ZZ_2}=\frac{2}{\bar{\rho}_\mathsf{Z}}\bigg(2\bar{\gamma}_\mathsf{Z}+\frac{\mu'_\mathsf{VE}\bar{\ell}'_\mathsf{E}}{\mu'_\mathsf{VE}-2}\bigg).$$

\subsubsection{Mean values corresponding to the volume measure}
\begin{thm}
The mean volume of the typical cell of the column tessellation, denoted by $\bar{\upsilon}_\mathsf{Z}$, is given as follows
$$\bar{\upsilon}_\mathsf{Z}=\frac{1}{\lambda'_\mathsf{Z}\bar{\rho}_\mathsf{Z}}.$$
\end{thm}
\begin{proof}
It is obvious from the fact that $\lambda_\mathsf{Z}\bar{\upsilon}_\mathsf{Z}=1$.
\end{proof}
The corresponding area and volume mean values of column tessellations of constant cell-height $1$ are
$$\bar{a}_\mathsf{P}=\frac{(\mu'_\mathsf{VE}-2)\bar{a}'_\mathsf{Z}+\mu'_\mathsf{VE}\bar{\ell}'_\mathsf{E}}{3\mu'_\mathsf{VE}-2},\quad\bar{a}_\mathsf{Z}=2\bigg(\bar{a}'_\mathsf{Z}+\frac{\mu'_\mathsf{VE}\bar{\ell}'_\mathsf{E}}{\mu'_\mathsf{VE}-2}\bigg),$$
$$\bar{a}_\mathsf{Z_2}=\frac{(\mu'_\mathsf{VE}-2)\bar{a}'_\mathsf{Z}+\mu'_\mathsf{VE}\bar{\ell}'_\mathsf{E}}{2\mu'_\mathsf{VE}-\phi-2},\quad\bar{a}_\mathsf{ZZ_2}=4\bar{a}'_\mathsf{Z}+\frac{2\mu'_\mathsf{VE}\bar{\ell}'_\mathsf{E}}{\mu'_\mathsf{VE}-2},\quad\bar{\upsilon}_\mathsf{Z}=\frac{1}{\lambda'_\mathsf{Z}}.$$

\end{document}